\documentclass{amsart}

\usepackage{amssymb,amsmath,amsthm,latexsym}
\usepackage{booktabs}
\usepackage[all]{xy}

\theoremstyle{definition}
\newtheorem{definition}{Definition}

\newtheorem{example}[definition]{Example}

\theoremstyle{plain}
\newtheorem{lemma}[definition]{Lemma}

\newtheorem{theorem}[definition]{Theorem}
\newtheorem{corollary}[definition]{Corollary}
\newtheorem{conjecture}[definition]{Conjecture}

\begin{document}

\title[Cayley's hyperdeterminant]
{Cayley's hyperdeterminant: a combinatorial approach via representation theory}

\author[Bremner]{Murray R. Bremner}
\author[Bickis]{Mikelis G. Bickis}
\author[Soltanifar]{Mohsen Soltanifar}

\address{Department of Mathematics and Statistics, University of Saskatchewan,
106 Wiggins Road (McLean Hall), Saskatoon, Saskatchewan, Canada S7N 5E6}

\email{bremner@math.usask.ca}
\email{bickis@math.usask.ca}
\email{mohsen.soltanifar@usask.ca}

\begin{abstract}
Cayley's hyperdeterminant is a homogeneous polynomial of degree 4 in the 8 entries of a $2 \times 2 \times 2$ array.
It is the simplest (nonconstant) polynomial which is invariant under changes of basis in three directions. 
We use elementary facts about representations of the 3-dimensional simple Lie algebra $\mathfrak{sl}_2(\mathbb{C})$ 
to reduce the problem of finding the invariant polynomials for a $2 \times 2 \times 2$ array to a combinatorial
problem on the enumeration of $2 \times 2 \times 2$ arrays with non-negative integer entries. 
We then apply results from linear algebra to obtain a new proof that Cayley's hyperdeterminant generates all the invariants.
\end{abstract}

\maketitle

%%%%%%%%%%%%%%%%%%%%%%%%%%%%%%%%%%%%%%%%%%%%%%%%%%%%%%%%%%%%%%%%%%%%%%%%

\section{Introduction}

In his famous 1845 paper on the theory of linear transformations,
which became the foundation of classical invariant theory, Cayley \cite{Cayley} introduced
the concept of the hyperdeterminant of a multidimensional array. He explicitly
calculated the hyperdeterminant for the simplest case, an array of size $2
\times 2 \times 2$, which can be represented in two dimensions by its
two frontal slices:
  \begin{equation}
  \label{arrayX}
  X
  =
  \left[
  \begin{array}{cc|cc}
  x_{000} & x_{010} & x_{001} & x_{011} \\
  x_{100} & x_{110} & x_{101} & x_{111}
  \end{array}
  \right].
  \end{equation}

\begin{definition} \label{CHdefinition}
\textbf{Cayley's hyperdeterminant} is the following homogeneous polynomial of degree 4
in the 8 entries $x_{ijk}$ of the $2 \times 2 \times 2$ array of
equation \eqref{arrayX}:
  \begin{align*}
  C
  &=
  x_{000}^2 x_{111}^2
  + x_{001}^2 x_{110}^2
  + x_{010}^2 x_{101}^2
  + x_{011}^2 x_{100}^2
  \\
  &\quad
  -
  2 \big(
  x_{000} x_{001} x_{110} x_{111}
  + x_{000} x_{010} x_{101} x_{111}
  + x_{000} x_{011} x_{100} x_{111}
  \\
  &\quad\quad\quad
  + x_{001} x_{010} x_{101} x_{110}
  + x_{001} x_{011} x_{100} x_{110}
  + x_{010} x_{011} x_{100} x_{101}
  \big)
  \\
  &\quad
  +
  4 \big(
  x_{000} x_{011} x_{101} x_{110}
  + x_{001} x_{010} x_{100} x_{111}
  \big).
  \end{align*}
\end{definition}

This polynomial has an interesting combinatorial-geometric interpretation.
The first four terms have coefficient 1, and the subscripts correspond to the vertices of diagonals of the cube
(configurations of dimension 1). 
The next six terms have coefficient $-2$, and the subscripts correspond to squares in the cube 
(configurations of dimension 2). 
The last two terms have coefficient 4, and the subscripts correspond to tetrahedra in the cube 
(configurations of dimension 3). 
These three configurations are illustrated by the dashed lines in Figure \ref{geometricconfigurations}.

\begin{figure}
$
\begin{xy}
( 0, 0)*+{\bullet}="000";
( 9, 4)*+{\circ}="002";
( 0,16)*+{\circ}="020";
(18, 0)*+{\circ}="200";
( 9,20)*+{\circ}="022";
(27, 4)*+{\circ}="202";
(18,16)*+{\circ}="220";
(27,20)*+{\bullet}="222";
{\ar@{-} "220";"020"};
{\ar@{-} "020";"000"};
{\ar@{-} "220";"200"};
{\ar@{-} "200";"000"};
{\ar@{-} "022";"020"};
{\ar@{-} "222";"220"};
{\ar@{-} "002";"000"};
{\ar@{-} "202";"200"};
{\ar@{-} "222";"022"};
{\ar@{-} "222";"202"};
{\ar@{-} "202";"002"};
{\ar@{-} "022";"002"};
{\ar@{--} "000";"222"};
(38, 0)*+{\bullet}="000";
(47, 4)*+{\bullet}="002";
(38,16)*+{\circ}="020";
(56, 0)*+{\circ}="200";
(47,20)*+{\circ}="022";
(65, 4)*+{\circ}="202";
(56,16)*+{\bullet}="220";
(65,20)*+{\bullet}="222";
{\ar@{-} "220";"020"};
{\ar@{-} "020";"000"};
{\ar@{-} "220";"200"};
{\ar@{-} "200";"000"};
{\ar@{-} "022";"020"};
{\ar@{--} "222";"220"};
{\ar@{--} "002";"000"};
{\ar@{-} "202";"200"};
{\ar@{-} "222";"022"};
{\ar@{-} "222";"202"};
{\ar@{-} "202";"002"};
{\ar@{-} "022";"002"};
{\ar@{--} "000";"220"};
{\ar@{--} "002";"222"};
( 76, 0)*+{\bullet}="000";
( 85, 4)*+{\circ}="002";
( 76,16)*+{\circ}="020";
( 94, 0)*+{\circ}="200";
( 85,20)*+{\bullet}="022";
(103, 4)*+{\bullet}="202";
( 94,16)*+{\bullet}="220";
(103,20)*+{\circ}="222";
{\ar@{-} "220";"020"};
{\ar@{-} "020";"000"};
{\ar@{-} "220";"200"};
{\ar@{-} "200";"000"};
{\ar@{-} "022";"020"};
{\ar@{-} "222";"220"};
{\ar@{-} "002";"000"};
{\ar@{-} "202";"200"};
{\ar@{-} "222";"022"};
{\ar@{-} "222";"202"};
{\ar@{-} "202";"002"};
{\ar@{-} "022";"002"};
{\ar@{--} "022";"220"};
{\ar@{--} "000";"202"};
{\ar@{--} "000";"220"};
{\ar@{--} "022";"202"};
{\ar@{--} "220";"202"};
{\ar@{--} "000";"022"}
\end{xy}
$
\caption{Geometric configurations in Cayley's hyperdeterminant}
\label{geometricconfigurations}
\end{figure}
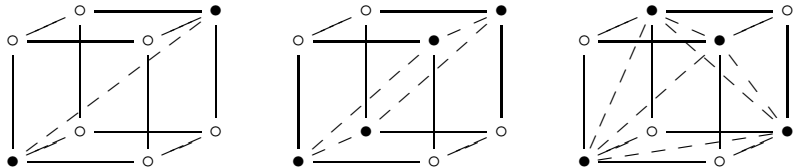

Cayley's hyperdeterminant $C$ is the simplest (nonconstant) polynomial in 
the entries of the $2 \times 2 \times 2$ array $X$ of equation
\eqref{arrayX} which is invariant under unimodular changes of basis along the
three directions. To make this
idea more precise, we regard $X$ as an element of the
tensor cube $\mathbb{C}^2 \otimes \mathbb{C}^2 \otimes \mathbb{C}^2$ of the
2-dimensional complex vector space $\mathbb{C}^2$. The group $SL_2(\mathbb{C})$ of $2 \times
2$ matrices of determinant 1 acts on $\mathbb{C}^2$ by matrix-vector
multiplication, and this gives a component-wise action of the direct product
$SL_2(\mathbb{C}) \times SL_2(\mathbb{C}) \times SL_2(\mathbb{C})$ on
$\mathbb{C}^2 \otimes \mathbb{C}^2 \otimes \mathbb{C}^2$.
This action
extends to the algebra of polynomials in the entries of $X$,
and $C$ is
the simplest polynomial which is fixed by every element of the
direct product. 

Ordinary determinants of square matrices can be characterized by a similar invariance property. 
Matrices $U \in SL_m(\mathbb{C})$ act on rectangular $m \times n$ matrices $A$ by left multiplication: 
$A \mapsto UA$. 
The First Fundamental Theorem of Classical Invariant Theory states that 
there exist nonconstant invariant polynomials in the entries of $A$ if and only if $m \le n$, 
and every invariant is a polynomial in the determinants of the $m \times m$ submatrices obtained 
by choosing $m$ columns of $A$; see Procesi \cite[\S 11.1.2]{Procesi}. 
If we combine the left action of $U \in SL_m(\mathbb{C})$ with the right action of $V \in
SL_n(\mathbb{C})$, so that $A \mapsto UAV$, then invariants exist for $SL_m(\mathbb{C}) \times SL_n(\mathbb{C})$ if and
only if $m = n$, and every invariant is a polynomial in $\det(A)$. 

We now summarize the results of this paper. In Section
\ref{sectionrepresentationtheory} we recall some elementary results in the
representation theory of the 3-dimensional simple Lie algebra
$\mathfrak{sl}_2(\mathbb{C})$. We explain how the 9-dimensional semisimple Lie algebra
  \[
  \mathfrak{sl}_2(\mathbb{C})^3
  =
  \mathfrak{sl}_2(\mathbb{C}) \oplus
  \mathfrak{sl}_2(\mathbb{C}) \oplus
  \mathfrak{sl}_2(\mathbb{C}),
  \]
acts on the 8-dimensional vector space
  \[
  M_{2,2,2}(\mathbb{C}) = \mathbb{C}^2 \otimes \mathbb{C}^2 \otimes \mathbb{C}^2,
  \]
the tensor cube of the natural representation of $\mathfrak{sl}_2(\mathbb{C})$.
We describe, using what are essentially the power and
product rules from elementary calculus, the action of
$\mathfrak{sl}_2(\mathbb{C})^3$ on the algebra of polynomials on
$M_{2,2,2}(\mathbb{C})$. The invariant polynomials are those which are
annihilated by all Lie algebra elements (equivalently, fixed by all Lie group
elements). For each degree $d$, the homogeneous polynomials form a
finite-dimensional representation of $\mathfrak{sl}_2(\mathbb{C})^3$, and a well-known theorem implies that this
representation is the direct sum of irreducible representations.
We express the invariant polynomials as the elements in the kernel of a
linear differential operator which represents the action of the Lie algebra on
homogeneous polynomials, and from this we represent the invariant
polynomials as the nullspace of a matrix. The domain of this
linear map has a monomial basis in bijection with the set of all $2
\times 2 \times 2$ arrays with non-negative integer entries summing to
$d$ and equal sums over the parallel $2 \times 2$ slices in
the three directions. This reduces the computation of invariants
to elementary combinatorics and linear algebra.

In Section \ref{sectionhyperdeterminant} we present explicit calculations for
degrees 2 and 4. In degree 2, the matrix has size $6 \times 4$ and rank 4, 
so there are no invariants. In
degree 4, the matrix has size $24 \times 12$ and rank 11; Cayley's
hyperdeterminant $C$ is a basis for the nullspace. Considering the powers of
$C$, it follows that the dimension of the space of invariants is $\ge 1$
in each degree $d$ which is a multiple of 4.

In Section \ref{sectiondimensionformulas} we compute the dimensions of certain weight spaces in the representation of
$\mathfrak{sl}_2(\mathbb{C})^3$ on the homogeneous polynomials of
degree $d$. This is equivalent to the enumeration of
$2 \times 2 \times 2$ arrays with non-negative integer entries and
constraints on the entry sums over the parallel $2 \times 2$ slices in
the three directions.

In Section \ref{sectioninclusionexclusion} we apply a result on subspaces, 
reminiscent of the inclusion-exclusion principle, to a commutative diagram of injective linear maps between
weight spaces in representations of $\mathfrak{sl}_2(\mathbb{C})^3$. This
provides a different proof that the space of invariant polynomials has dimension $\ge 1$ 
in each degree $d$ which is a multiple of 4. 
We then use the representation theory of $\mathfrak{sl}_2(\mathbb{C})^3$ to prove that 
the algebra of invariants is a polynomial algebra and is
generated by Cayley's hyperdeterminant in degree 4.
Hence there are no new invariants in higher degrees.

In Section \ref{sectiongeneralization} we consider invariant polynomials in the entries of an array of size
$n_1 \times n_2 \times \cdots \times n_k$ under the action of 
$SL_{n_1}(\mathbb{C}) \times SL_{n_2}(\mathbb{C}) \times \cdots \times SL_{n_k}(\mathbb{C})$. 
The corresponding combinatorial objects are $k$-dimensional arrays with non-negative integer entries and
equal sums over the parallel slices in the $k$ directions.

In Section \ref{sectionconclusion} we briefly summarize recent applications of
Cayley's hyperdeterminant and provide some suggestions for further research.

%%%%%%%%%%%%%%%%%%%%%%%%%%%%%%%%%%%%%%%%%%%%%%%%%%%%%%%%%%%%%%%%%%%%%%%%

\section{Representations of $\mathfrak{sl}_2(\mathbb{C})$} \label{sectionrepresentationtheory}

In this section we recall some elementary results in the representation theory of Lie algebras. 
Standard references are Jacobson \cite{Jacobson},
Humphreys \cite{Humphreys}, de Graaf \cite{deGraaf}, Erdmann and Wildon \cite{ErdmannWildon}.
For an introduction to Lie theory, by which is meant the relation between Lie groups and Lie algebras, 
see Stillwell \cite{Stillwell}.
For the connection with classical invariant theory, see Procesi \cite{Procesi}.

The 3-dimensional simple Lie algebra $\mathfrak{sl}_2(\mathbb{C})$ consists of
the $2 \times 2$ matrices of trace 0 over $\mathbb{C}$ with the Lie bracket operation
$[ A, B ] = AB - BA$.
This operation satisfies anticommutativity and the Jacobi identity:
  \[
  [A,A] \equiv 0,
  \qquad
  [[A,B],C] + [[B,C],A] + [[C,A],B] \equiv 0.
  \]
The standard basis of $\mathfrak{sl}_2(\mathbb{C})$ consists of these three matrices:
  \begin{equation} \label{sl2basis}
  H = \left[ \begin{array}{rr} 1 & 0 \\ 0 & -1 \end{array} \right],
  \qquad
  E = \left[ \begin{array}{rr} 0 & 1 \\ 0 & 0 \end{array} \right],
  \qquad
  F = \left[ \begin{array}{rr} 0 & 0 \\ 1 & 0 \end{array} \right].
  \end{equation}
In its natural representation, $\mathfrak{sl}_2(\mathbb{C})$ acts by matrix-vector multiplication
on the two-dimensional vector space $\mathbb{C}^2$ with this standard basis:
  \[
  x_0 = \left[ \begin{array}{r} 1 \\ 0 \end{array} \right],
  \qquad
  x_1 = \left[ \begin{array}{r} 0 \\ 1 \end{array} \right].
  \]

\begin{lemma}
The action of $\mathfrak{sl}_2(\mathbb{C})$ on $\mathbb{C}^2$ is given by the following equations:
  \[
  H \cdot x_0 = x_0,
  \;\;
  H \cdot x_1 = -x_1.
  \;\;
  E \cdot x_0 = 0,
  \;\;
  E \cdot x_1 = x_0,
  \;\;
  F \cdot x_0 = x_1,
  \;\;
  F \cdot x_1 = 0.
  \]
In particular, $x_0$ and $x_1$ are eigenvectors for $H$ and $H \cdot x_i = (-1)^i x_i$ $(i = 0, 1)$.
\end{lemma}

\begin{lemma} \label{differentialoperators}
If we regard $x_0$ and $x_1$ as indeterminates,
then we can express the action of $\mathfrak{sl}_2(\mathbb{C})$ on $\mathbb{C}^2$
by partial differential operators as follows:
  \[
  H = x_0 \frac{\partial}{\partial x_0} - x_1 \frac{\partial}{\partial x_1},
  \qquad
  E = x_0 \frac{\partial}{\partial x_1},
  \qquad
  F = x_1 \frac{\partial}{\partial x_0}.
  \]
\end{lemma}

\begin{definition}
We identify a $2 \times 2 \times 2$ array $X = ( x_{ijk} )$ with an element of the tensor cube
$M_{2,2,2}(\mathbb{C}) = \mathbb{C}^2 \otimes \mathbb{C}^2 \otimes \mathbb{C}^2$.
We identify the entries with simple tensors,
$x_{ijk} = x_i \otimes x_j \otimes x_k$ $(i, j, k = 0, 1)$.
(Strictly speaking, since we regard $x_{ijk}$ as a coordinate function on $M_{2,2,2}(\mathbb{C})$,
we should use dual basis vectors and write
$x_{ijk} = x^\ast_i \otimes x^\ast_j \otimes x^\ast_k$,
but this distinction will not be important for us.)
\end{definition}

\begin{definition}
The Lie group $SL_2(\mathbb{C}) \times SL_2(\mathbb{C}) \times SL_2(\mathbb{C})$ acts on
the vector space $M_{2,2,2}(\mathbb{C})$; the action is defined on simple tensors and extended linearly:
  \[
  ( X, Y, Z ) \cdot ( u \otimes v \otimes w )
  =
  ( X \cdot u ) \otimes ( Y \cdot v ) \otimes ( Z \cdot w ).
  \]
As usual, we linearize the group action by considering the action of the Lie algebra
$\mathfrak{sl}_2(\mathbb{C})^3 =
\mathfrak{sl}_2(\mathbb{C}) \oplus \mathfrak{sl}_2(\mathbb{C}) \oplus \mathfrak{sl}_2(\mathbb{C})$
on $M_{2,2,2}(\mathbb{C})$ defined by this equation:
  \[
  ( A, B, C ) \cdot ( u \otimes v \otimes w )
  =
  ( A \cdot u ) \otimes v \otimes w
  +
  u \otimes ( B \cdot v ) \otimes w
  +
  u \otimes v \otimes ( C \cdot w ).
  \]
\end{definition}

\begin{lemma}
The 8-dimensional vector space $M_{2,2,2}(\mathbb{C})$ is an irreducible representation of
the 9-dimensional semisimple Lie algebra $\mathfrak{sl}_2(\mathbb{C})^3$.
\end{lemma}

\begin{proof}
A representation
of a semisimple Lie algebra is irreducible if and only if it is isomorphic to
the tensor product of irreducible representations of its simple summands.  
See Proposition 1.1 of Neher, Savage and Senesi
\cite{NSS}.
\end{proof}

\begin{definition}
We write $H_\ell, E_\ell, F_\ell$ $(\ell = 1, 2, 3)$ for the standard basis
of the $\ell$-th copy of $\mathfrak{sl}_2(\mathbb{C})$ in $\mathfrak{sl}_2(\mathbb{C})^3$;
see equation \eqref{sl2basis}.
\end{definition}

\begin{lemma}
The basis of $\mathfrak{sl}_2(\mathbb{C})^3$
acts on the basis of $M_{2,2,2}(\mathbb{C})$ as follows:
  \begin{alignat*}{3}
  H_1 \cdot x_{0jk} &= x_{0jk},
  &\qquad
  H_2 \cdot x_{i0k} &= x_{i0k},
  &\qquad
  H_3 \cdot x_{ij0} &= x_{ij0},
  \\
  H_1 \cdot x_{1jk} &= -x_{1jk},
  &\qquad
  H_2 \cdot x_{i1k} &= -x_{i1k},
  &\qquad
  H_3 \cdot x_{ij1} &= -x_{ij1},
  \\
  E_1 \cdot x_{0jk} &= 0,
  &\qquad
  E_2 \cdot x_{i0k} &= 0,
  &\qquad
  E_3 \cdot x_{ij0} &= 0,
  \\
  E_1 \cdot x_{1jk} &= x_{0jk},
  &\qquad
  E_2 \cdot x_{i1k} &= x_{i0k},
  &\qquad
  E_3 \cdot x_{ij1} &= x_{ij0},
  \\
  F_1 \cdot x_{0jk} &= x_{1jk},
  &\qquad
  F_2 \cdot x_{i0k} &= x_{i1k},
  &\qquad
  F_3 \cdot x_{ij0} &= x_{ij1},
  \\
  F_1 \cdot x_{1jk} &= 0,
  &\qquad
  F_2 \cdot x_{i1k} &= 0,
  &\qquad
  F_3 \cdot x_{ij1} &= 0.
  \end{alignat*}
\end{lemma}

\begin{proof}
This is a straightforward calculation.  For example, for $\ell = 1$ we have
  \begin{align*}
  H_1 \cdot x_{0jk}
  &=
  H_1 \cdot ( x_0 \otimes x_j \otimes x_k )
  =
  ( H_1 \cdot x_0 ) \otimes x_j \otimes x_k
  =
  x_{0jk},
  \\
  H_1 \cdot x_{1jk}
  &=
  H_1 \cdot ( x_1 \otimes x_j \otimes x_k )
  =
  ( H_1 \cdot x_1 ) \otimes x_j \otimes x_k
  =
  - x_{1jk},
  \\
  E_1 \cdot x_{0jk}
  &=
  E_1 \cdot ( x_0 \otimes x_j \otimes x_k )
  =
  ( E_1 \cdot x_0 ) \otimes x_j \otimes x_k
  =
  0,
  \\
  E_1 \cdot x_{1jk}
  &=
  E_1 \cdot ( x_1 \otimes x_j \otimes x_k )
  =
  ( E_1 \cdot x_1 ) \otimes x_j \otimes x_k
  =
  x_{0jk},
  \\
  F_1 \cdot x_{0jk}
  &=
  F_1 \cdot ( x_0 \otimes x_j \otimes x_k )
  =
  ( F_1 \cdot x_0 ) \otimes x_j \otimes x_k
  =
  x_{1jk},
  \\
  F_1 \cdot x_{1jk}
  &=
  F_1 \cdot ( x_1 \otimes x_j \otimes x_k )
  =
  ( F_1 \cdot x_1 ) \otimes x_j \otimes x_k
  =
  0.
  \end{align*}
The other cases are similar.
\end{proof}

\begin{definition}
We consider the polynomial algebra on $M_{2,2,2}(\mathbb{C})$:
  \[
  P
  =
  \mathbb{C}[ x_{000}, x_{010}, x_{100}, x_{110}, x_{001}, x_{011}, x_{101}, x_{111} ].
  \]
\end{definition}

\begin{lemma}
A basis of $P$ over $\mathbb{C}$ consists of the monomials,
  \[
  \prod_{i,j,k=0,1} x_{ijk}^{e_{ijk}}
  =
  x_{000}^{e_{000}}
  x_{001}^{e_{001}}
  x_{010}^{e_{010}}
  x_{011}^{e_{011}}
  x_{100}^{e_{100}}
  x_{101}^{e_{101}}
  x_{110}^{e_{110}}
  x_{111}^{e_{111}},
  \]
where the exponents $e_{ijk}$ are arbitrary non-negative integers.
\end{lemma}

\begin{definition}
The \textbf{degree of a monomial} is the sum of its exponents:
  \[
  d = \sum_{i,j,k=0,1} e_{ijk}.
  \]
We write $P_d$ for the homogeneous subspace of $P$ spanned by the monomials of degree $d$.
We identify $P_1$ with $M_{2,2,2}(\mathbb{C})$, so that a basis of $P_1$ consists of the monomials of degree 1, 
namely
$x_{000}$, $x_{001}$, $x_{010}$, $x_{011}$, $x_{100}$, $x_{101}$, $x_{110}$, $x_{111}$.
\end{definition}

\begin{lemma}
There is a bijection between the monomials of degree $d$ and
the $2 \times 2 \times 2$ arrays $E = (e_{ijk})$ of non-negative integers summing to $d$.
\end{lemma}

\begin{lemma}
The polynomial algebra $P$ is graded by the degree:
  \[
  S = \bigoplus_{d \ge 0} P_d,
  \qquad
  P_d P_e \subseteq P_{d+e}.
  \]
\end{lemma}

\begin{lemma}
We have $P_d = S^d P_1$, the $d$-th symmetric power of $P_1$.
The action of an element $D \in \mathfrak{sl}_2(\mathbb{C})^3$
extends to all basis monomials of $P$ by the derivation rule
$D \cdot ( f g ) = ( D \cdot f ) g + f ( D \cdot g )$.
It follows by induction that
  \[
  D \cdot x_{ijk}^{e_{ijk}}
  =
  e_{ijk} x_{ijk}^{e_{ijk}-1} ( D \cdot x_{ijk} ),
  \]
and hence that  
  \begin{align*}
  D \cdot \prod_{i,j,k} x_{ijk}^{e_{ijk}}
  &=
  \sum_{i',j',k'}
  x_{000}^{e_{000}}
  \cdots
  \big( D \cdot x_{i'j'k'}^{e_{i'j'k'}} \big)
  \cdots
  x_{111}^{e_{111}}
  \\
  &=
  \sum_{i',j',k'}
  x_{000}^{e_{000}}
  \cdots
  \Big( e_{i'j'k'} x_{i'j'k'}^{e_{i'j'k'}-1} ( D \cdot x_{i'j'k'} ) \Big)
  \cdots
  x_{111}^{e_{111}}
  \end{align*}
In particular, $D \cdot P_d \subseteq P_d$ for all $D \in \mathfrak{sl}_2(\mathbb{C})^3$.
\end{lemma}

\begin{lemma}
For every $d \ge 1$, the subspace $P_d$ is a finite-dimensional representation of $\mathfrak{sl}_2(\mathbb{C})^3$,
and is therefore isomorphic to a direct sum of irreducible representations.
\end{lemma}

\begin{lemma}
For every non-negative integer $n$, there is (up to isomorphism) a unique
irreducible representation of $\mathfrak{sl}_2(\mathbb{C})$ with dimension
$n{+}1$, denoted $V(n)$. This representation is generated by a vector $v_n$
with $H \cdot v_n = n v_n$. With respect to the basis $v_{n-2i}$ ($i = 0, 1,
\dots, n$), the action of $\mathfrak{sl}_2(\mathbb{C})$ on $V(n)$ is given by
  \begin{alignat*}{2}
  H \cdot v_{n-2i} &= (n{-}2i) v_{n-2i},
  \\
  E \cdot v_{n-2i} &= (n{-}i{+}1) v_{n-2i+2}
  \quad
  (i = 1, 2, \dots, n),
  &\qquad
  E \cdot v_n &= 0,
  \\
  F \cdot v_{n-2i} &= (i{+}1) v_{n-2i-2}
  \quad
  (i = 0, 1, \dots, n{-}1),
  &\qquad
  F \cdot v_{-n} &= 0.
  \end{alignat*}
\end{lemma}

\begin{definition}
In the irreducible representation $V(n)$ of $\mathfrak{sl}_2(\mathbb{C})$,
the basis vector $v_{n-2i}$ is a \textbf{weight vector} (that is, $H$-eigenvector) of weight $n{-}2i$.
\end{definition}

\begin{lemma}
An irreducible representation of $\mathfrak{sl}_2(\mathbb{C})^3$
is isomorphic to the tensor product $V(a) \otimes V(b) \otimes V(c)$
for some non-negative integers $a, b, c$.
\end{lemma}

\begin{lemma}
The polynomials invariant under the group
$SL_2(\mathbb{C})^3$
coincide with the polynomials annihilated by the Lie algebra
$\mathfrak{sl}_2(\mathbb{C})^3$.
\end{lemma}

\begin{lemma}
A polynomial $f \in P_d$ is invariant if and only if
$D \cdot f = 0$ for all $D \in \mathfrak{sl}_2(\mathbb{C})^3$.
Equivalently, the invariant polynomials correspond to the summands of $P_d$
isomorphic to $V(0) \otimes V(0) \otimes V(0)$.
Therefore, a polynomial $f \in P_d$ is invariant if and only if
$H_\ell \cdot f = 0$ and $E_\ell \cdot f = 0$ for $\ell = 1, 2, 3$.
\end{lemma}

\begin{lemma}
The basis monomial
  \[
  \prod_{i,j,k} x_{ijk}^{e_{ijk}},
  \]
is a simultaneous eigenvector for $H_1, H_2, H_3$ with eigenvalues
  \[
  \sum_{j,k} e_{0jk} - \sum_{j,k} e_{1jk},
  \qquad
  \sum_{i,k} e_{i0k} - \sum_{i,k} e_{i1k},
  \qquad
  \sum_{i,j} e_{ij0} - \sum_{i,j} e_{ij1}.
  \]
\end{lemma}

\begin{proof}
For $\ell = 1$ we have
  \[
  H_1 \cdot x_{ijk}^{e_{ijk}}
  =
  e_{ijk} x_{ijk}^{e_{ijk}-1} \big( H_1 \cdot x_{ijk} \big)
  =
  e_{ijk} x_{ijk}^{e_{ijk}-1} (-1)^i x_{ijk}
  =
  (-1)^i e_{ijk} x_{ijk}^{e_{ijk}},
  \]
and therefore
  \[
  H_1 \cdot \prod_{i,j,k} x_{ijk}^{e_{ijk}}
  =
  \Big( \sum_{j,k} e_{0jk} - \sum_{j,k} e_{1jk} \Big)
  \prod_{i,j,k} x_{ijk}^{e_{ijk}}.
  \]
The other two cases are similar.
\end{proof}

\begin{definition}
The \textbf{weight space} $W(d;a,b,c)$ is the subspace of $P_d$
spanned by the monomials which have eigenvalues $a$, $b$, $c$ for $H_1$, $H_2$, $H_3$ respectively.
The \textbf{zero weight space} is $W(d;0,0,0)$.
\end{definition}

\begin{lemma}
The basis monomial
  \[
  \prod_{i,j,k} x_{ijk}^{e_{ijk}},
  \]
belongs to $W(d;0,0,0)$ if and only if
  \[
  \sum_{i,j,k} e_{ijk} = n,
  \quad
  \sum_{j,k} e_{0jk} = \sum_{j,k} e_{1jk},
  \quad
  \sum_{i,k} e_{i0k} = \sum_{i,k} e_{i1k},
  \quad
  \sum_{i,j} e_{ij0} = \sum_{i,j} e_{ij1}.
  \]
That is, the $2 \times 2 \times 2$ array $( e_{ijk} )$ of exponents
satisfies the condition that in each of the three directions,
the parallel $2 \times 2$ slices have equal sums.
\end{lemma}

\begin{lemma} \label{lemmaevendegree}
If $d$ is odd then the zero weight space $W(d;0,0,0)$ is the zero subspace.
In particular, there are no invariant polynomials in odd degrees.
\end{lemma}

\begin{lemma}
The actions of $E_1, E_2, E_3$ induce these linear maps on weight spaces:
  \begin{align*}
  &E_1 \colon W(d;0,0,0) \to W(d;2,0,0),
  \\
  &E_2 \colon W(d;0,0,0) \to W(d;0,2,0),
  \\
  &E_3 \colon W(d;0,0,0) \to W(d;0,0,2).
  \end{align*}
\end{lemma}

\begin{definition}
We define a linear map
  \[
  \mathcal{E}_d \colon
  W(d;0,0,0) \longrightarrow W(d;2,0,0) \oplus W(d;0,2,0) \oplus W(d;0,0,2),
  \]
by the equation $\mathcal{E}_d(f) = ( \, E_1 \cdot f, \, E_2 \cdot f, \, E_3 \cdot f \, )$
for all $f \in W(d;0,0,0)$.
\end{definition}

\begin{lemma}
The invariant polynomials in $P_d$ coincide with the kernel of $\mathcal{E}_d$.
\end{lemma}

We can represent the linear map $\mathcal{E}_d$ by the matrix $[\mathcal{E}_d]$
with respect to the ordered monomial bases of the weight spaces.
The size $[\mathcal{E}_d]$ is
  \[
  \big( \dim W(d;2,0,0) + \dim W(d;0,2,0) + \dim W(d;0,0,2) \, \big)
  \times
  \dim W(d;0,0,0).
  \]
In fact the three dimensions in parentheses are equal; this follows by considering the automorphisms
of $\mathfrak{sl}_2(\mathbb{C})^3$ which permute the three summands.

%%%%%%%%%%%%%%%%%%%%%%%%%%%%%%%%%%%%%%%%%%%%%%%%%%%%%%%%%%%%%%%%%%%%%%%%

\section{Cayley's hyperdeterminant via linear algebra} \label{sectionhyperdeterminant}

In this section we show by direct calculation that every (nonconstant) invariant polynomial of degree $\le 4$
is a scalar multiple of Cayley's hyperdeterminant.

We identify monomials with sequences of exponents lexicographically ordered by
their triples of subscripts:
  \[
  \prod_{i,j,k} x_{ijk}^{e_{ijk}}  
  \longleftrightarrow
  [
  e_{000},
  e_{001},
  e_{010},
  e_{011},
  e_{100},
  e_{101},
  e_{110},
  e_{111}
  ].
  \]
Within each weight space, we order the basis monomials lexicographically.

\begin{lemma}
There are no invariant polynomials in degree 2.
\end{lemma}

\begin{proof}
A basis of the zero weight space $W(2;0,0,0)$ consists of four monomials,
  \[
  00011000 \qquad 
  00100100 \qquad 
  01000010 \qquad 
  10000001.
  \] 
which label the columns of the matrix $[\mathcal{E}_2]$.
Each nonzero weight space $W(2;0,0,2)$, $W(2;0,2,0)$, $W(2;2,0,0)$ has a basis of two monomials
which label the rows of $[\mathcal{E}_2]$:
  \[
  \begin{array}{c}
  01100000 \\
  10010000 \\ \midrule
  01001000 \\
  10000100 \\ \midrule
  00101000 \\
  10000010 
  \end{array}
  \qquad\qquad
  [\mathcal{E}_2]
  =
  \left[
  \begin{array}{cccc}
  0 & 1 & 1 & 0 \\
  1 & 0 & 0 & 1 \\ \midrule
  1 & 0 & 1 & 0 \\
  0 & 1 & 0 & 1 \\ \midrule
  1 & 1 & 0 & 0 \\
  0 & 0 & 1 & 1
  \end{array}
  \right]
  \]
The matrix has full rank, and so its nullspace is $\{0\}$.
\end{proof}

  \begin{figure}
  $
  \begin{array}{cccccc}
  00022000 &\; 
  00111100 &\; 
  00200200 &\;
  01011010 &\; 
  01100110 &\; 
  01101001
  \\ 
  02000020 &\; 
  10010110 &\; 
  10011001 &\; 
  10100101 &\; 
  11000011 &\; 
  20000002
  \end{array}
  $
  \caption{Monomial basis for zero weight space $W(4;0,0,0)$}
  \label{basisdomain}
  \bigskip
  \[
  \begin{array}{c}
  01111000 \\ 
  01200100 \\ 
  02100010 \\ 
  10021000 \\ 
  10110100 \\ 
  11010010 \\ 
  11100001 \\ 
  20010001 \\ \midrule
  01012000 \\ 
  01101100 \\ 
  02001010 \\ 
  10011100 \\ 
  10100200 \\ 
  11000110 \\ 
  11001001 \\ 
  20000101 \\ \midrule
  00112000 \\ 
  00201100 \\ 
  01101010 \\ 
  10011010 \\ 
  10100110 \\ 
  10101001 \\ 
  11000020 \\ 
  20000011 
  \end{array}
  \qquad \qquad
  \left[
  \begin{array}{cccccccccccc}
  . & 1 & . & 1 & . & 1 & . & . & . & . & . & . \\ 
  . & . & 2 & . & 1 & . & . & . & . & . & . & . \\ 
  . & . & . & . & 1 & . & 2 & . & . & . & . & . \\ 
  2 & . & . & . & . & . & . & . & 1 & . & . & . \\ 
  . & 1 & . & . & . & . & . & 1 & . & 1 & . & . \\ 
  . & . & . & 1 & . & . & . & 1 & . & . & 1 & . \\ 
  . & . & . & . & . & 1 & . & . & . & 1 & 1 & . \\ 
  . & . & . & . & . & . & . & . & 1 & . & . & 2 \\ \midrule 
  2 & . & . & 1 & . & . & . & . & . & . & . & . \\ 
  . & 1 & . & . & 1 & 1 & . & . & . & . & . & . \\ 
  . & . & . & 1 & . & . & 2 & . & . & . & . & . \\ 
  . & 1 & . & . & . & . & . & 1 & 1 & . & . & . \\ 
  . & . & 2 & . & . & . & . & . & . & 1 & . & . \\ 
  . & . & . & . & 1 & . & . & 1 & . & . & 1 & . \\ 
  . & . & . & . & . & 1 & . & . & 1 & . & 1 & . \\ 
  . & . & . & . & . & . & . & . & . & 1 & . & 2 \\ \midrule
  2 & 1 & . & . & . & . & . & . & . & . & . & . \\ 
  . & 1 & 2 & . & . & . & . & . & . & . & . & . \\ 
  . & . & . & 1 & 1 & 1 & . & . & . & . & . & . \\ 
  . & . & . & 1 & . & . & . & 1 & 1 & . & . & . \\ 
  . & . & . & . & 1 & . & . & 1 & . & 1 & . & . \\ 
  . & . & . & . & . & 1 & . & . & 1 & 1 & . & . \\ 
  . & . & . & . & . & . & 2 & . & . & . & 1 & . \\ 
  . & . & . & . & . & . & . & . & . & . & 1 & 2
  \end{array}
  \right]
  \]
  \caption{The matrix $[\mathcal{E}_4]$}
  \label{degree4matrix}
  \bigskip
  \[
  \left[
  \begin{array}{rrrrrrrrrrrr}
  1 & . & . & . & . & . & . & . & . & . & . & -1 \\
  . & 1 & . & . & . & . & . & . & . & . & . &  2 \\
  . & . & 1 & . & . & . & . & . & . & . & . & -1 \\
  . & . & . & 1 & . & . & . & . & . & . & . &  2 \\
  . & . & . & . & 1 & . & . & . & . & . & . &  2 \\
  . & . & . & . & . & 1 & . & . & . & . & . & -4 \\
  . & . & . & . & . & . & 1 & . & . & . & . & -1 \\
  . & . & . & . & . & . & . & 1 & . & . & . & -4 \\
  . & . & . & . & . & . & . & . & 1 & . & . &  2 \\
  . & . & . & . & . & . & . & . & . & 1 & . &  2 \\
  . & . & . & . & . & . & . & . & . & . & 1 &  2
  \end{array}
  \right]
  \]
  \caption{The row canonical form of $[\mathcal{E}_4]$}
  \label{degree4matrixrcf}
  \end{figure}

\begin{theorem}
In degree 4, the space of invariant polynomials has dimension 1;
every invariant is a scalar multiple of Cayley's hyperdeterminant $C$.
\end{theorem}

\begin{proof}
A basis of the zero weight space $W(4;0,0,0)$ consists of the 12 monomials in Figure \ref{basisdomain}.
Each nonzero weight space $W(4;0,0,2)$, $W(4;0,2,0)$, $W(4;2,0,0)$ has a basis of 8 monomials; 
see Figure \ref{degree4matrix}, which also displays 
the $24 \times 12$ matrix $[\mathcal{E}_4]$ (we use dot for zero).
Figure \ref{degree4matrixrcf} gives the row canonical form of $[\mathcal{E}_4]$
(we omit zero rows).
The rank is 11, and Cayley's hyperdeterminant is a basis of the nullspace.
\end{proof}

\begin{corollary} \label{corollarylowerbound}
The dimension of the space of invariant polynomials is at least 1
in each degree $d$ congruent to 0 modulo 4.
\end{corollary}

\begin{proof}
The existence of Cayley's hyperdeterminant $C$ in degree 4
implies that there is at least one invariant polynomial $C^e$ in each degree $d = 4e$.
\end{proof}

%%%%%%%%%%%%%%%%%%%%%%%%%%%%%%%%%%%%%%%%%%%%%%%%%%%%%%%%%%%%%%%%%%%%%%%%

\section{Dimension formulas for weight spaces} \label{sectiondimensionformulas}

Our next goal is to prove that there are no new invariants in higher degrees;
in other words, that every invariant is a polynomial in $C$.
To do this, we need to prove that the lower bound of Corollary \ref{corollarylowerbound}
is also an upper bound.
The first step is to obtain dimension formulas for certain weight spaces
in the representation of $\mathfrak{sl}_2(\mathbb{C})^3$ on
the space $P_d$ of homogeneous polynomials of degree $d$.

\begin{theorem} \label{theoremdimensions}
The dimension of the zero weight subspace $W(d;0,0,0)$ equals
  \allowdisplaybreaks
  \begin{alignat}{2}
  &\frac{1}{384} ( d + 4 )^2 ( d^2 + 8 d + 24 )
  &\qquad
  &\text{if $d \equiv 0$ (mod 4)},
  \label{000-0} \tag{000-0}
  \\
  &\frac{1}{384} ( d + 2 ) ( d + 6 ) ( d^2 + 8 d + 28 )
  &\qquad
  &\text{if $d \equiv 2$ (mod 4)}.
  \label{000-2} \tag{000-2}
  \\
\intertext{The dimensions of $W(d;2,0,0)$, $W(d;0,2,0)$ and $W(d;0,0,2)$ equal}
  &\frac{1}{384} d ( d + 4 )^2 ( d + 8 )
  &\qquad
  &\text{if $d \equiv 0$ (mod 4)},
  \label{200-0} \tag{200-0}
  \\
  &\frac{1}{384} ( d + 2 ) ( d + 6 ) ( d^2 + 8 d + 4 )
  &\qquad
  &\text{if $d \equiv 2$ (mod 4)}.
  \label{200-2} \tag{200-2}
  \\
\intertext{The dimensions of $W(d;2,2,0)$, $W(d;2,0,2)$ and $W(d;0,2,2)$ equal}
  &\frac{1}{384} d ( d + 4 ) ( d^2 + 12 d + 8 )
  &\qquad
  &\text{if $d \equiv 0$ (mod 4)},
  \label{220-0} \tag{220-0}
  \\
  &\frac{1}{384} ( d + 2 ) ( d^3 + 14 d^2 + 28 d - 24 )
  &\qquad
  &\text{if $d \equiv 2$ (mod 4)}.
  \label{220-2} \tag{220-2}
  \\
\intertext{The dimension of $W(d;2,2,2)$ equals}
  &\frac{1}{384} d ( d^3 + 16 d^2 + 32 d + 32 )
  &\qquad
  &\text{if $d \equiv 0$ (mod 4)},
  \label{222-0} \tag{222-0}
  \\
  &\frac{1}{384} ( d + 2 ) ( d^3 + 14 d^2 + 4 d + 24 )
  &\qquad
  &\text{if $d \equiv 2$ (mod 4)}.
  \label{222-2} \tag{222-2}
  \end{alignat}
In all cases, the dimension is 0 if $d$ is odd.
\end{theorem}

Given non-negative integers $d$ (the degree) and $a, b, c$ (the weights),
we consider $2 \times 2 \times 2$ arrays $E = ( e_{ijk} )$ $(i,j,k \in \{0,1\})$
of non-negative integer exponents satisfying the following equations:
  \begin{align}
  e_{000} + e_{001} + e_{010} + e_{011} + e_{100} + e_{101} + e_{110} + e_{111}
  &=
  d,
  \label{D} \tag{D}
  \\
  ( e_{000} + e_{001} + e_{010} + e_{011} ) - ( e_{100} + e_{101} + e_{110} + e_{111} )
  &=
  a,
  \label{M1} \tag{M1}
  \\
  ( e_{000} + e_{001} + e_{100} + e_{101} ) - ( e_{010} + e_{011} + e_{110} + e_{111} )
  &=
  b,
  \label{M2} \tag{M2}
  \\
  ( e_{000} + e_{010} + e_{100} + e_{110} ) - ( e_{001} + e_{011} + e_{101} + e_{111} )
  &=
  c.
  \label{M3} \tag{M3}
  \end{align}
These equations hold if and only if the corresponding monomial belongs to the weight space $W(d;a,b,c)$;
that is, the number of arrays $E$ satisfying equations \eqref{D}--\eqref{M3} equals the dimension of $W(d;a,b,c)$.
Theorem \ref{theoremdimensions} gives formulas for these dimensions for certain values of $a, b, c$.
These formulas are polynomials of degree 4, as expected since we have eight exponents and four constraints.

\begin{lemma} \label{onedim}
Consider $2 \times 2$ matrices $( e_{ij} )$ with non-negative integer entries and
specified row sums $r_0, r_1$ and column sums $c_0, c_1$ satisfying $r_0 + r_1 = c_0 + c_1$:
  \begin{equation}
  \left[
  \begin{array}{cc}
  e_{00} & e_{01} \\
  e_{10} & e_{11}
  \end{array}
  \right],
  \qquad
  \begin{array}{l}
  e_{00} + e_{01} = r_0, \\
  e_{10} + e_{11} = r_1,
  \end{array}
  \qquad
  \begin{array}{l}
  e_{00} + e_{10} = c_0, \\
  e_{01} + e_{11} = c_1.
  \end{array}
  \label{one-slice}
  \end{equation}
The number of such matrices equals $\min( r_0, r_1, c_0, c_1 ) + 1$.
\end{lemma}

\begin{proof}
We have four variables and four constraints, but one dependence relation among the constraints,
so we expect a 1-dimensional solution set.
Without loss of generality, we can interchange the rows (resp.~columns)
and assume that $r_0 \le r_1$ (resp.~$c_0 \le c_1$);
we can also transpose the matrix and assume that $r_0 \le c_0$.
It is clear that since $c_1-r_0 \ge c_1-c_0 \ge 0$ we have the particular solution
  \[
  \left[
  \begin{array}{cc}
  e_{00} & e_{01} \\
  e_{10} & e_{11}
  \end{array}
  \right]
  =
  \left[
  \begin{array}{cc}
  0 & r_0 \\
  c_0 & c_1-r_0
  \end{array}
  \right].
  \]
If $u$ is any integer then we can preserve the constraints by adding $u$ to the diagonal entries
and subtracting $u$ from the off-diagonal entries:
  \[
  \left[
  \begin{array}{cc}
  e_{00} & e_{01} \\
  e_{10} & e_{11}
  \end{array}
  \right]
  =
  \left[
  \begin{array}{cc}
  u & r_0-u\\
  c_0-u & c_1-r_0+u
  \end{array}
  \right].
  \]
This is another solution if and only if $0 \le u \le r_0$.
Hence the number of solutions is $r_0 + 1 = \min( r_0, r_1, c_0, c_1 ) + 1$.
\end{proof}

\begin{lemma} \label{sum-min-square}
For any integer $k \ge 1$ we have
  \[
  \sum_{i=1}^k \sum_{j=1}^k \big( \min(i,j) \big)^2 
  = 
  \frac16 k ( k + 1 ) ( k^2 + k + 1 ).
  \]
\end{lemma}

\begin{proof}
By induction on $k$; the result is clear for $k = 1$.
We have
  \begin{align*}
  &
  \sum_{i=1}^{k+1} \sum_{j=1}^{k+1} \big( \min(i,j) \big)^2
  \\
  &=
  \sum_{i=1}^k \sum_{j=1}^k \big( \min(i,j) \big)^2
  +
  \sum_{i=1}^{k+1} \big( \min( i, k{+}1 ) \big)^2
  +
  \sum_{j=1}^k \big( \min( k{+}1, j ) \big)^2
  \\
  &=
  \frac16 k (k{+}1) (k^2{+}k{+}1)
  +
  \frac16 (k{+}1) (k{+}2) (2k{+}3)
  +
  \frac16 k (k{+}1) (2k{+}1)
  \\
  &=
  \frac16 ( k + 1 ) ( k + 2 ) ( k^2 + 3k + 3 ),
  \end{align*}
using the formula for the sum of the squares from 1 to $k{+}1$.
\end{proof}

We now come to the proof of Theorem \ref{theoremdimensions}.
We prove the first two equations \eqref{000-0} and \eqref{000-2}; 
the proofs of the others are similar but slightly more complicated,
and the details are not particularly enlightening.

\begin{proof}
Equations \eqref{M1}--\eqref{M3} imply that $d$ is even, since if $a = b = c = 0$ then
each of the sums in parentheses equals $d/2$.
Hence we assume that $d = 2m$.

In equations \eqref{D}--\eqref{M3} we write $w, x, y, z$ for the row and column sums
of the $2 \times 2$ slice $( e_{0jk} )$ with $i = 0$.
We then have $w + x = m$ and $y + z = m$, and
  \begin{alignat*}{2}
  e_{000} + e_{001} &= w,
  &\qquad \qquad
  e_{100} + e_{101} &= m-w,
  \\
  e_{010} + e_{011} &= x,
  &\qquad \qquad
  e_{110} + e_{111} &= m-x,
  \\
  e_{000} + e_{010} &= y,
  &\qquad \qquad
  e_{100} + e_{110} &= m-y,
  \\
  e_{001} + e_{011} &= z,
  &\qquad\qquad
  e_{101} + e_{111} &= m-z.
  \end{alignat*}
Suppose that $w \le x$ and $y \le z$.
Lemma \ref{onedim} shows that
  \begin{itemize}
  \item
  the number of $2 \times 2$ slices $( e_{0jk} )$ is $\min(w,y) + 1$, and
  \item
  the number of $2 \times 2$ slices $( e_{1jk} )$ is $\min( m{-}x, m{-}z ) + 1 = \min(w,y) + 1$.
  \end{itemize}
Hence the number of $2 \times 2 \times 2$ arrays is $( \min(w,y) + 1 )^2$.
Any solution with $w < x$ has a corresponding solution with $w > x$
obtained by interchanging the slices $( e_{i0k} )$ and $( e_{i1k} )$.
Any solution with $y < z$ has a corresponding solution with $y > z$
obtained by interchanging $( e_{ij0} )$ and $( e_{ij1} )$.

We first prove equation \eqref{000-2}: the case $d \equiv 2$ (mod 4).
We have $m = 2k{-}1$ where $k = (d{+}2)/4$.
Since $m$ is odd, we cannot have either $w = x$ or $y = z$; hence all solutions are doubly paired.
Thus the number of solutions is four times the number with $w < x$ and $y < z$,
and for this we apply Lemma \ref{sum-min-square}:
  \begin{align*}
  &
  4
  \sum_{w=0}^{k-1} \sum_{y=0}^{k-1} \big( \min(w,y) + 1 \big)^2
  =
  4
  \sum_{w=0}^{k-1} \sum_{y=0}^{k-1} \min( w{+}1, y{+}1 )^2
  \\
  = \;
  &
  \frac23 k (k+1) (k^2+k+1)
  =
  \frac{1}{384} (d+2)(d+6)(d^2+8d+28).
  \end{align*}
We next prove equation \eqref{000-0}: the case $d \equiv 0$ (mod 4).
We have $m = 2k$ where $k = d/4$.
In this case we must also consider $w = x$ and $y = z$, so we add
  \begin{align*}
  &
  2 \sum_{w=0}^{k-1} \min( w{+}1, k{+}1 )^2
  +
  2 \sum_{y=0}^{k-1} \min( k{+}1, y{+}1 )^2
  +
  \min( k{+}1, k{+}1 )^2
  \\
  = \;
  &
  \frac23 k (k+1) (2k+1)
  +
  (k+1)^2
  =
  \frac13 (k+1)(4k^2+5k+3),
  \end{align*}
to the previous result, obtaining
  \[
  \frac23 k (k+1) (k^2+k+1)
  +
  \frac13 (k+1)(4k^2+5k+3)
  =
  \frac{1}{384} (d+4)^2 (d^2+8d+24).
  \]
This completes the proof.
\end{proof}

%%%%%%%%%%%%%%%%%%%%%%%%%%%%%%%%%%%%%%%%%%%%%%%%%%%%%%%%%%%%%%%%%%%%%%%%

\section{Inclusion-exclusion for subspaces}
\label{sectioninclusionexclusion}

We recall a familiar formula from elementary linear algebra. If $U_1$ and $U_2$
are finite-dimensional subspaces of a vector space then
  \begin{equation} \label{case n=2}
  \dim( \, U_1 + U_2 \, )
  =
  \dim( \, U_1 \, ) + \dim( \, U_2 \, ) - \dim( \, U_1 \cap U_2 \, ).
  \end{equation}
The next result generalizes equation \eqref{case n=2} to an arbitrary finite number of subspaces, 
and is similar to the combinatorial formula for inclusion-exclusion on finite sets.

\begin{lemma} \label{dimensioninequality}
If $U_1, \dots, U_n$ are finite-dimensional subspaces of a vector space then
  \[
  \dim\Big( \, \sum_{i=1}^n U_i \, \Big)
  \; \le \;
  \sum_{r=1}^n (-1)^{r+1}
  \!\!\!\!\!\!
  \sum_{1 \le i_1 < \cdots < i_r \le n}
  \!\!\!\!\!\!
  \dim( \, U_{i_1} \cap \cdots \cap U_{i_r} \, ),
  \]
where the inner sum on the right is over all $\binom{n}{r}$ subsets $\{ i_1,
\dots, i_r \} \subseteq \{ 1, \dots, n \}$.
\end{lemma}

\begin{proof}
The statement is false if the inequality is replaced by an equality:
consider three distinct lines through the origin in the plane.
The proof is by induction on $n$.  The statement is clear for $n \le 2$. We
assume the statement for $n$ and prove it for $n{+}1$. Using equation 
\eqref{case n=2} we obtain
  \begin{align*}
  \dim\Big( \, \sum_{i=1}^{n+1} U_i \, \Big)
  &=
  \dim\Big( \, \big( \sum_{i=1}^n U_i \big) + U_{n+1} \, \Big)
  \\
  &=
  \dim\Big( \sum_{i=1}^n U_i \Big)
  +
  \dim( U_{n+1} )
  -
  \dim\Big( \, \big( \sum_{i=1}^n U_i \big) \cap U_{n+1} \, \Big).
  \end{align*}
Observing that
  \[
  \dim\Big( \, \big( \sum_{i=1}^n U_i \big) \cap U_{n+1} \, \Big)
  \; \ge \;
  \dim\Big( \, \sum_{i=1}^n \big( \, U_i \cap U_{n+1} \, \big) \, \Big),
  \]
we obtain
  \[
  \dim\Big( \, \sum_{i=1}^{n+1} U_i \, \Big)
  \; \le \;
  \dim\Big( \sum_{i=1}^n U_i \Big)
  +
  \dim( U_{n+1} )
  -
  \dim\Big( \, \sum_{i=1}^n \big( \, U_i \cap U_{n+1} \, \big) \, \Big).
  \]
We apply the inductive hypothesis to the last two sums, obtaining
  \begin{align*}
  &
  \sum_{r=1}^n (-1)^{r+1}
  \!\!\!\!\!\!
  \sum_{1 \le i_1 < \cdots < i_r \le n}
  \!\!\!\!\!\!
  \dim( \, U_{i_1} \cap \cdots \cap U_{i_r} \, )
  +
  \dim( U_{n+1} )
  \\
  &
  -
  \sum_{r=1}^n (-1)^{r+1}
  \!\!\!\!\!\!
  \sum_{1 \le i_1 < \cdots < i_r \le n}
  \!\!\!\!\!\!
  \dim\big( \, ( U_{i_1} \cap U_{n+1} ) \cap \cdots \cap ( U_{i_r} \cap U_{n+1} ) \, \big).
  \end{align*}
We separate the $r = 1$ terms of the first double sum, 
and simplify the second double sum using familiar properties of intersections:
  \begin{align*}
  &
  \sum_{1 \le i \le n}
  \dim( \, U_i \, )
  +
  \sum_{r=2}^n (-1)^{r+1}
  \!\!\!\!\!\!
  \sum_{1 \le i_1 < \cdots < i_r \le n}
  \!\!\!\!\!\!
  \dim( \, U_{i_1} \cap \cdots \cap U_{i_r} \, )
  +
  \dim( U_{n+1} )
  \\
  &
  -
  \sum_{r=1}^n (-1)^{r+1}
  \!\!\!\!\!\!
  \sum_{1 \le i_1 < \cdots < i_r \le n}
  \!\!\!\!\!\!
  \dim\big( U_{i_1} \cap \cdots \cap U_{i_r} \cap U_{n+1} \, ).
  \end{align*}
A slight rearrangement gives
  \begin{align*}
  &
  \sum_{1 \le i \le n+1}
  \dim( \, U_i \, )
  +
  \sum_{r=2}^n (-1)^{r+1}
  \!\!\!\!\!\!
  \sum_{1 \le i_1 < \cdots < i_r \le n}
  \!\!\!\!\!\!
  \dim( \, U_{i_1} \cap \cdots \cap U_{i_r} \, )
  \\
  &
  +
  \sum_{r=1}^n (-1)^{(r+1)+1}
  \!\!\!\!\!\!
  \sum_{1 \le i_1 < \cdots < i_r \le n}
  \!\!\!\!\!\!
  \dim\big( \, ( U_{i_1} \cap \cdots \cap U_{i_r} ) \cap U_{n+1} \, \big).
  \end{align*}
The first (resp. second) double sum corresponds to the subsets of size $r$
(resp. size $r{+}1$) of the set $\{1,\dots,n{+}1\}$ which exclude (resp.
include) $n{+}1$, so we obtain
  \[
  \sum_{r=1}^{n+1} (-1)^{r+1}
  \!\!\!\!\!\!
  \sum_{1 \le i_1 < \cdots < i_r \le n+1}
  \!\!\!\!\!\!
  \dim( \, U_{i_1} \cap \cdots \cap U_{i_r} \, ),
  \]
and this completes the proof.
\end{proof}

We now consider a reformulation of this problem, in which we have a positive integer $n$
and a collection of $2^n$ finite-dimensional vector spaces,
  \[
  \{ \, V_{i_1,i_2,\dots,i_n} \mid 0 \le i_1, i_2, \dots ,i_n \le 1 \, \},
  \]
corresponding to the vertices of an $n$-dimensional cube. 
We also have $n 2^{n-1}$ injective linear maps corresponding to the
edges of the cube,
  \[
  f^{(k)}_{i_1,\dots,\widehat{i_k},\dots,i_n} \colon
  V_{i_1,\dots,1,\dots,i_n} \longrightarrow V_{i_1,\dots,0,\dots,i_n},
  \]
where the hat indicates omission and the values of the indices are
  \[
  1 \le k \le n,
  \qquad
  (i_1,\dots,\widehat{i_k},\dots,i_n) \in \{0,1\}^{n-1}.
  \]
Given any two of these vector spaces, we assume that all compositions of linear
maps between the spaces give the same result; that is, the diagram is commutative.
We can therefore identify each space $V_{i_1,i_2,\dots,i_n}$ with
its image in $V_{0,0,\dots,0}$, and so all of the spaces
$V_{i_1,i_2,\dots,i_n}$ can be identified with subspaces of $V_{0,0,\dots,0}$.

We define $n$ vector spaces $U_1, \dots, U_n$ by starting at the vertex $(0,\dots,0)$ 
of the $n$-dimensional cube and following the $n$ edges to the vertices
  \[
  U_i = V_{0,\dots,1,\dots,0}
  \qquad
  (1 \le i \le n),
  \]
in which the subscripts on the right are 0 except for 1 in position $i$.
Given any $r$-element subset $\{ i_1, \dots, i_r \} \subseteq \{ 1, \dots, n \}$,
we write $\chi( i_1, \dots, i_r )$ for the element of $\{0,1\}^n$ which has 1 in 
positions $i_1, \dots, i_r$ and 0 elsewhere.  Our assumptions 
allow us to make the following identifications: 
  \[
  U_{i_1} \cap \cdots \cap U_{i_r}
  =
  V_{\chi( i_1, \dots, i_r )}.
  \]
Lemma \ref{dimensioninequality} then implies that
  \begin{equation} \label{newinequality}
  \begin{array}{l}
  \dim\Big( \, 
  \mathrm{im}\big( f_{0,\dots,0}^{(1)} \big)
  +
  \cdots
  +
  \mathrm{im}\big( f_{0,\dots,0}^{(n)} \big)
  \, \Big)
  \; \le \;
  \\
  \displaystyle{
  \sum_{r=1}^n (-1)^{r+1}
  \!\!\!\!\!\!
  \sum_{1 \le i_1 < \cdots < i_r \le n}
  \!\!\!\!\!\!
  \dim\big( \, V_{\chi( i_1, \dots, i_r )} \, \big).
  }
  \end{array}
  \end{equation}

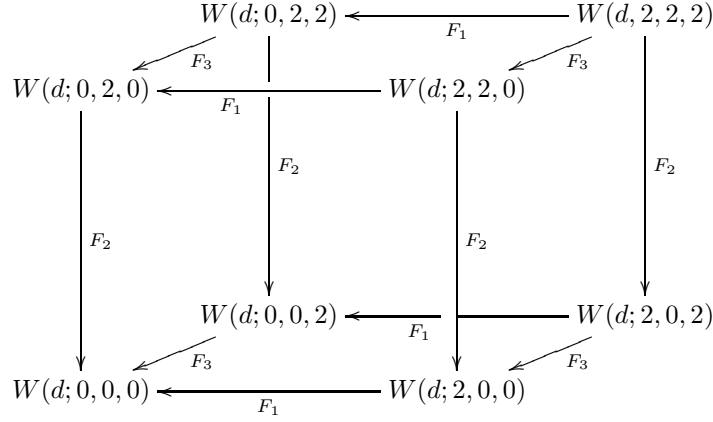
\begin{figure}
$
\begin{xy}
( 0, 0)*+{W(d;0,0,0)}="000";
(25,10)*+{W(d;0,0,2)}="002";
( 0,40)*+{W(d;0,2,0)}="020";
(50, 0)*+{W(d;2,0,0)}="200";
(25,50)*+{W(d;0,2,2)}="022";
(75,10)*+{W(d;2,0,2)}="202";
(50,40)*+{W(d;2,2,0)}="220";
(75,50)*+{W(d,2,2,2)}="222";
{\ar^(0.6){F_1} "220";"020"};
{\ar^{F_2} "020";"000"};
{\ar^{F_2} "220";"200"};
{\ar^{F_1} "200";"000"};
{\ar^(0.4){F_3} "022";"020"};
{\ar^(0.4){F_3} "222";"220"};
{\ar^(0.4){F_3} "002";"000"};
{\ar^(0.4){F_3} "202";"200"};
{\ar^{F_1} "222";"022"};
{\ar^{F_2} "222";"202"};
{\ar@{->}|>>>>>>>>>>>>>>{\hole}^(0.6){F_1} "202";"002"};
{\ar@{->}|>>>>>>>>>>>>>>>>>>>>>>>>>>>{\hole}^{F_2} "022";"002"}
\end{xy}
$
\caption{Linear maps among weight spaces in degree $d$}
\label{inclusionexclusion}
\end{figure}

\begin{example} \label{mainexample}
We consider $n = 3$ and identify the 8 vertices of the cube with the following weight spaces
in degree $d$ defined in Section \ref{sectiondimensionformulas}:
  \begin{align*}
  &
  W(d;0,0,0), \quad
  W(d;2,0,0), \quad
  W(d;0,2,0), \quad
  W(d;0,0,2),
  \\
  &
  W(d;2,2,0), \quad
  W(d;2,0,2), \quad
  W(d;0,2,2), \quad
  W(d;2,2,2).
  \end{align*}
The representation theory of $\mathfrak{sl}_2(\mathbb{C})$ shows that
the action of the basis elements $F_1, F_2, F_3$ on the homogeneous polynomials
of degree $d$ gives injective linear
maps between these weight spaces as illustrated in Figure \ref{inclusionexclusion}.
The invariant polynomials are the nonzero elements in the irreducible
summands $V(0) \otimes V(0) \otimes V(0)$, and the number of these summands
equals the codimension, in the zero weight space $W(d;0,0,0)$,
of the sum of the images of the weight spaces 
$W(d;2,0,0)$, $W(d;0,2,0)$, $W(d;0,0,2)$
under the actions of $F_1, F_2, F_3$ respectively.
That is,
  \begin{itemize}
  \item We start with the entire zero weight zero space $W(d;0,0,0)$.
  \item We factor out the images of vectors of weight (2,0,0) or (0,2,0) or
      (0,0,2) by the action of $F_1$ or $F_2$ or $F_3$.
  \item The vectors that come from weight
      (2,2,0) or (2,0,2) or (0,2,2) by the action of $F_1, F_2$ or $F_1,
      F_3$ or $F_2, F_3$ have then been factored out twice, so we must add those dimensions back in.
   \item But then the vectors that come from weight (2,2,2) by
       the action of $F_1, F_2, F_3$ must be factored out again.
   \end{itemize}
The dimension formulas from Section \ref{sectiondimensionformulas}
with equation \eqref{newinequality} give
  \begin{align*}
  &
  \dim W(d;0,0,0) - \dim W(d;2,0,0) - \dim W(d;0,2,0) - \dim W(d;0,0,2)
  \\
  &
  + \dim W(d;2,2,0) + \dim W(d;2,0,2) + \dim W(d;0,2,2) - \dim W(d;2,2,2)
  \\
  &=
  \begin{cases}
  1 &\text{if $n \equiv 0$ (mod 4)} \\
  0 &\text{otherwise}.
  \end{cases}
  \end{align*}
Combining this with Lemma \ref{dimensioninequality}, 
this gives another proof of Corollary \ref{corollarylowerbound}:
the dimension of the space of invariants is $\ge 1$ in degrees $d \equiv 0$ (mod 4).
\end{example}

\begin{theorem}
Every polynomial in the entries $x_{ijk}$ of the $2 \times 2 \times 2$ array $X = (x_{ijk})$ $(i, j, k = 0, 1)$,
which is invariant under changes of basis with determinant 1 along all the three directions,
is a polynomial in Cayley's hyperdeterminant.
\end{theorem}

\begin{proof}
It remains to use the representation theory of Lie algebras to show that 
inequality \eqref{newinequality} becomes in fact an equality in the situation 
of Example \ref{mainexample}.
We know that the space $P_d$ of homogeneous polynomials of degree $d$ is 
completely reducible as a representation of the semisimple Lie algebra
$\mathfrak{sl}_2(\mathbb{C})^3$, and that the irreducible summands 
are tensor products $V(a) \otimes V(b) \otimes V(c)$ of
irreducible representations of $\mathfrak{sl}_2(\mathbb{C})$.  Since the
weight spaces in the tensor factors have dimension 1 as representations
of $\mathfrak{sl}_2(\mathbb{C})$, it follows that the weight spaces in the 
tensor product have dimension 1 as representations of $\mathfrak{sl}_2(\mathbb{C})^3$.
Inequality \eqref{newinequality} is obviously an equality when all the dimensions
are 1, and this completes the proof.
\end{proof}

%%%%%%%%%%%%%%%%%%%%%%%%%%%%%%%%%%%%%%%%%%%%%%%%%%%%%%%%%%%%%%%%%%%%%%%%

\section{General multidimensional arrays} \label{sectiongeneralization}

We consider a $k$-dimensional array of size $n_1 \times n_2 \times \cdots \times n_k$: 
  \[
  X = ( x_{i_1 i_2 \cdots i_k} )
  \qquad
  (
  1 \le i_1 \le n_1, \;
  1 \le i_2 \le n_2, \;
  \dots, \;
  1 \le i_k \le n_k).
  \]
(The smallest index is now 1, not 0.)
We consider an extension of determinants to these arrays,
using a combinatorial approach based on the representation theory of the special linear Lie algebra
$\mathfrak{sl}_n(\mathbb{C})$.
As usual we write $\mathbb{C}^{n_1}, \mathbb{C}^{n_2}, \dots, \mathbb{C}^{n_k}$
for the complex vector spaces with dimensions $n_1, n_2, \dots, n_k$ and standard bases
  \[
  e^{(1)}_{i_1} \, (i_1 = 1, \dots, n_1), \quad
  e^{(2)}_{i_2} \, (i_2 = 1, \dots, n_2), \quad
  \dots, \quad
  e^{(k)}_{i_k} \, (i_k = 1, \dots, n_k).
  \]
A tensor of order $k$ is an element of the tensor product
  \[
  \mathbb{C}^{n_1,n_2,\dots,n_k}
  =
  \mathbb{C}^{n_1} \otimes \mathbb{C}^{n_2} \otimes \cdots \otimes \mathbb{C}^{n_k}.
  \]

\begin{lemma}
Every element of $\mathbb{C}^{n_1,n_2,\dots,n_k}$ is a finite sum of elements of the form
  \[
  v_1 \otimes v_2 \otimes \cdots \otimes v_k
  \quad
  (v_1 \in \mathbb{C}^{n_1}, v_2 \in \mathbb{C}^{n_2}, \dots, v_k \in \mathbb{C}^{n_k}).
  \]
A basis for $\mathbb{C}^{n_1,n_2,\dots,n_k}$ over $\mathbb{C}$
consists of the $n_1 n_2 \cdots n_k$ simple tensors
  \[
  e_{i_1,i_2,\dots,i_k}
  =
  e^{(1)}_{i_1} \otimes
  e^{(2)}_{i_2} \otimes
  \cdots \otimes
  e^{(k)}_{i_k}.
  \]
Every tensor of order $k$ can be expressed uniquely in the form
  \[
  \sum_{i_1}^{n_1}
  \sum_{i_2}^{n_2}
  \cdots
  \sum_{i_k}^{n_k}
  x_{i_1,i_2,\dots,i_k}
  e_{i_1,i_2,\dots,i_k}
  \quad
  (x_{i_1,i_2,\dots,i_k} \in \mathbb{C}).
  \]
\end{lemma}

A $k$-dimensional array consists of the coefficients
of a tensor of order $k$ with respect to the basis of simple tensors:
  \[
  X = ( x_{i_1,i_2,\dots,i_k} )
  \qquad
  ( i_1 = 1, \dots, n_1; \, i_2 = 1, \dots, n_2; \, \dots; \, i_k = 1, \dots, n_k ).
  \]
If $M_1, M_2, \dots, M_k$ are linear operators on $\mathbb{C}^{n_1}, \mathbb{C}^{n_2}, \dots \mathbb{C}^{n_k}$
then, with respect to the standard bases, we identify $M_\ell$ with
an $n_\ell \times n_\ell$ matrix for $\ell = 1, 2, \dots, k$:
  \[
  M_\ell = \big( m^{(\ell)}_{ij} \big)
  \qquad
  ( m^{(\ell)}_{ij} \in \mathbb{C}; \,
  i, j = 1, \dots, n_\ell ).
  \]
The action of a $k$-tuple of operators $M = (M_1, M_2, \dots, M_k)$ on a simple tensor in 
$\mathbb{C}^{n_1,n_2,\dots,n_k}$ is given by the equation
  \begin{equation}
  \label{actionsimpletensor}
  ( M_1, M_2, \dots, M_k ) \cdot ( v_1 \otimes v_2 \otimes \cdots \otimes v_k )
  =
  M_1 v_1 \otimes M_2 v_2 \otimes \cdots \otimes M_k v_k.
  \end{equation}
We introduce $n_1 n_2 \cdots n_k$ indeterminates corresponding to the entries of $X$:
  \[
  x_{i_1,i_2,\dots,i_k}
  \qquad
  (i_1 = 1, \dots, n_1; \; i_2 = 1, \dots, n_2; \; \dots; \; i_k = 1, \dots, n_k).
  \]
We consider the polynomial algebra in these indeterminates over $\mathbb{C}$:
  \[
  \mathbb{C}[ \, x_{i_1,i_2,\dots,i_k}
  \mid
  i_1 = 1, \dots, n_1; \,
  i_2 = 1, \dots, n_2; \,
  \dots; \,
  i_k = 1, \dots, n_k \, ].
  \]
For $\ell = 1, 2, \dots, k$ the action of $M_\ell$ on an indeterminate
corresponds to its action on the standard basis vectors in $\mathbb{C}^{n_\ell}$:
  \begin{equation}
  \label{actionindeterminate}
  M_\ell \, e^{(\ell)}_j = \sum_{i=1}^{n_\ell} m^{(\ell)}_{ij} e^{(\ell)}_i
  \implies
  M_\ell \cdot x_{j_1,\dots,j_\ell,\dots,j_k}
  =
  \sum_{i=1}^{n_\ell} m^{(\ell)}_{ij_\ell} x_{i_1,\dots,i,\dots,i_k}.
  \end{equation}
From this we obtain the action of $M = (M_1, M_2, \dots, M_k)$ on an indeterminate:
  \[
  ( M_1, M_2, \dots, M_k ) \cdot x_{j_1,j_2,\dots,j_k}
  =
  \sum_{i_1=1}^{n_1}
  \sum_{i_2=1}^{n_2}
  \cdots
  \sum_{i_k=1}^{n_k}
  m^{(1)}_{i_1 j_1}
  m^{(2)}_{i_2 j_2}
  \cdots
  m^{(k)}_{i_k j_k}
  x_{i_1,i_2,\dots,i_k}.
  \]
This action of $M = ( M_1, M_2, \dots, M_k )$ extends to an action on polynomials:
  \begin{align*}
  &
  M \cdot
  f\big( x_{1 1 \dots 1}, \dots, x_{j_1 j_2 \dots j_k}, \dots, x_{n_1 n_2 \dots n_k} \big)
  =
  \\
  &
  f\big( M \cdot x_{1 1 \dots 1}, \dots, M \cdot x_{j_1 j_2 \dots j_k}, \dots, M \cdot x_{n_1 n_2 \dots n_k} \big).
  \end{align*}

\begin{definition}
The polynomial $f \in \mathbb{C}[ \, x_{i_1,i_2,\dots,i_k}]$ is \textbf{invariant} if
  \[
  \det(M_\ell) = 1 \; (\ell = 1, \dots, k) 
  \implies 
  M \cdot f = f, \; M = ( M_1, M_2, \dots, M_k ).
  \]
\end{definition}

The $n \times n$ complex matrices of determinant 1, with the usual operation of matrix multiplication,
form the special linear group $SL_n(\mathbb{C})$.
Finite-dimensional representations of $SL_n(\mathbb{C})$ can be studied in terms of 
the Lie algebra $\mathfrak{sl}_n(\mathbb{C})$,
which consists of all $n \times n$ complex matrices of trace 0;
the bilinear product is the Lie bracket $[A,B] = AB - BA$.
The standard basis of $\mathfrak{sl}_n(\mathbb{C})$ consists of
  \begin{itemize}
  \item
the matrix units $U_{i,j}$ for $i \ne j$ with $(i,j)$ entry 1 and other entries 0,
  \item
the diagonal matrices $H_i = U_{i,i} - U_{i+1,i+1}$ for $i = 1, 2, \dots, n{-}1$.
  \end{itemize}
The simple root vectors are the matrix units $E_i = U_{i,i+1}$ for $i = 1, 2, \dots, n{-}1$.
The natural representation of $\mathfrak{sl}_n(\mathbb{C})$ is its action on $\mathbb{C}^n$ 
by matrix-vector multiplication.

\begin{lemma} \label{natural}
In the natural representation of $\mathfrak{sl}_n(\mathbb{C})$ we have
  \begin{align*}
  H_i \cdot e_j
  =
  \begin{cases}
  e_j &\text{if $j = i$} \\
  -e_j &\text{if $j = i{+}1$} \\
  0 &\text{otherwise},
  \end{cases}
  \qquad \qquad
  E_i \cdot e_j
  =
  \begin{cases}
  e_{j-1} &\text{if $j = i{+}1$} \\
  0 &\text{otherwise}.
  \end{cases}
  \end{align*}
\end{lemma}

We consider the action of the semisimple Lie algebra
  \begin{equation} \label{semisimple}
  \bigoplus_{\ell=1}^k \mathfrak{sl}_{n_\ell}(\mathbb{C})
  =
  \mathfrak{sl}_{n_1}(\mathbb{C}) \oplus
  \mathfrak{sl}_{n_2}(\mathbb{C}) \oplus
  \cdots \oplus
  \mathfrak{sl}_{n_k}(\mathbb{C}),
  \end{equation}
on its irreducible representation $\mathbb{C}^{n_1,n_2,\dots,n_k}$,
the tensor product of the natural representations of its simple summands.
For $\ell = 1, 2, \dots, k$ we write $H^{(\ell)}_i$, $E^{(\ell)}_i$ for
the elements $H_i$, $E_i \in \mathfrak{sl}_{n_\ell}(\mathbb{C})$.
Combining equations \eqref{actionsimpletensor} and \eqref{actionindeterminate} with
Lemma \ref{natural} we obtain the action of $H^{(\ell)}_i$ and $E^{(\ell)}_i$ on the indeterminates
$x_{j_1 j_2 \dots j_k}$.

\begin{lemma}
For $\ell = 1, 2, \dots, k$ and $i = 1, 2, \dots, n_\ell{-}1$ we have
  \begin{align*}
  H^{(\ell)}_i \cdot x_{j_1, j_2, \dots, j_k}
  &=
  \begin{cases}
  x_{j_1, j_2, \dots, j_k} &\text{if $j_\ell = i$} \\
  -x_{j_1, j_2, \dots j_k} &\text{if $j_\ell = i{+}1$} \\
  0 &\text{otherwise},
  \end{cases}
  \\
  E^{(\ell)}_i \cdot x_{j_1, j_2, \dots, j_k}
  &=
  \begin{cases}
  x_{j_1, j_2, \dots, j_\ell-1, \dots, j_k} &\text{if $j_\ell = i{+}1$} \\
  0 &\text{otherwise}.
  \end{cases}
  \end{align*}
\end{lemma}

The action of a Lie algebra $L$ on a tensor product $V \otimes W$
of representations is given by the derivation rule:
  \[
  x \cdot ( v \otimes w ) = ( x \cdot v ) \otimes w + v \otimes ( x \cdot w )
  \qquad
  ( x \in L, \, v \in V, \, w \in W ).
  \]
We identify the $d$-th symmetric power $S^d V$ of the representation $V$
with the space of homogeneous polynomials of degree $d$ on a basis of $V$.
It follows by induction on $d$ that the action of $L$ on $S^d V$ is given by
the following equation:
  \begin{align*}
  &
  x \cdot ( v_1^{e_1} v_2^{e_2} \cdots v_p^{e_p} )
  =
  \sum_{i=1}^p v_1^{e_1} \cdots ( x \cdot v_i^{e_i} ) \cdots v_p^{e_p}
  \\
  = \;
  &
  \sum_{i=1}^p v_1^{e_1} \cdots \big( e_i v_i^{e_i-1} ( x \cdot v_i ) \big) \cdots v_p^{e_p}
  =
  \sum_{i=1}^p e_i \, v_1^{e_1} \cdots v_i^{e_i-1} \cdots v_p^{e_p} ( x \cdot v_i ).
  \end{align*}
We apply this to
  \[
  L
  =
  \bigoplus_{\ell=1}^k \mathfrak{sl}_{n_\ell}(\mathbb{C}),
  \qquad
  V
  =
  \bigoplus_{j_1=1}^{n_1} \bigoplus_{j_2=1}^{n_2} \cdots \bigoplus_{j_k=1}^{n_k}
  \mathbb{C} x_{j_1 j_2 \dots j_k}.
  \]
Some equations will be clearer if we write a monomial as follows:
  \[
  \prod_{j_1}^{n_1} \prod_{j_2}^{n_2} \cdots \prod_{j_k}^{n_k}
  x_{j_1 j_2 \dots j_k}^{e_{j_1 j_2 \dots j_k}}
  =
  x_{1 \dots 1}^{e_{1 \dots 1}}
  \cdots
  x_{j_1 \dots j_k}^{e_{j_1 \dots j_k}}
  \cdots
  x_{n_1 \dots n_k}^{e_{n_1 \dots n_k}}
  \]

\begin{lemma}
For $\ell = 1, 2, \dots, k$ and $i = 1, 2, \dots, n_\ell{-}1$ we have
  \begin{align*}
  &
  H^{(\ell)}_i \cdot
  \big(
  x_{1 \dots 1}^{e_{1 \dots 1}}
  \cdots
  x_{j_1 \dots j_k}^{e_{j_1 \dots j_k}}
  \cdots
  x_{n_1 \dots n_k}^{e_{n_1 \dots n_k}}
  \big)
  =
  \\
  &
  \sum_{j_1=1}^{n_1} \cdots \sum_{j_k=1}^{n_k}
  \big( \delta_{j_\ell,i} - \delta_{j_\ell,i+1} \big)
  e_{j_1 \dots j_\ell \dots j_k}
  \,
  x_{1 \dots 1}^{e_{1 \dots 1}}
  \cdots
  x_{j_1 \dots j_\ell \dots j_k}^{e_{j_1 \dots j_\ell \dots j_k}}
  \cdots
  x_{n_1 \dots n_k}^{e_{n_1 \dots n_k}},
  \\
  &
  E^{(\ell)}_i \cdot
  \big(
  x_{1 \dots 1}^{e_{1 \dots 1}}
  \cdots
  x_{j_1 \dots j_k}^{e_{j_1 \dots j_k}}
  \cdots
  x_{n_1 \dots n_k}^{e_{n_1 \dots n_k}}
  \big)
  =
  \\
  &
  \sum_{j_1=1}^{n_1} \cdots \sum_{j_k=1}^{n_k}
  \delta_{j_\ell,i+1}
  \,
  e_{j_1 \dots j_\ell \dots j_k}
  \,
  x_{1 \dots 1}^{e_{1 \dots 1}}
  \cdots
  x_{j_1 \dots j_\ell-1 \dots j_k}^{e_{j_1 \dots j_\ell-1 \dots j_k}+1}
  \cdots
  x_{j_1 \dots j_\ell \dots j_k}^{e_{j_1 \dots j_\ell \dots j_k}-1}
  \cdots
  x_{n_1 \dots n_k}^{e_{n_1 \dots n_k}},
  \end{align*}
where $\delta_{ij}$ is the Kronecker delta ($\delta_{ii} = 1$, $\delta_{ij} = 0$ for $i \ne j$).
\end{lemma}

\begin{lemma} \label{weightlemma}
For every $\ell = 1, 2, \dots, k$ and $i = 1, 2, \dots, n_\ell{-}1$,
the monomial
  \[
  x_{1 \dots 1}^{e_{1 \dots 1}} \cdots x_{n_1 \dots n_k}^{e_{n_1 \dots n_k}},
  \]
is an eigenvector for $H^{(\ell)}_i$ with eigenvalue
  \[
  \sum_{j_1=1}^{n_1} \cdots \widehat{\sum_{j_\ell}} \cdots \sum_{j_k=1}^{n_k}
  e_{j_1 \dots i \dots j_k}
  -
  \sum_{j_1=1}^{n_1} \cdots \widehat{\sum_{j_\ell}} \cdots \sum_{j_k=1}^{n_k}
  e_{j_1 \dots i+1 \dots j_k},
  \]
where the hat denotes omission.
\end{lemma}

The space of homogeneous polynomials of degree $d$ has the basis
  \[
  x_{1 \dots 1}^{e_{1 \dots 1}} \cdots x_{n_1 \dots n_k}^{e_{n_1 \dots n_k}},
  \qquad
  \sum_{j_1=1}^{n_1} \cdots \sum_{j_k=1}^{n_k}
  e_{j_1 \dots j_k}
  =
  d.
  \]

\begin{definition} \label{zerodefinition}
A monomial
$x_{1 \dots 1}^{e_{1 \dots 1}} \cdots x_{n_1 \dots n_k}^{e_{n_1 \dots n_k}}$
has \textbf{weight zero} if it has eigenvalue 0 for every $H^{(\ell)}_i$
with $\ell = 1, 2, \dots, k$ and $i = 1, 2, \dots, n_\ell{-}1$;
that is,
  \[
  \sum_{j_1=1}^{n_1} \cdots \widehat{\sum_{j_\ell}} \cdots \sum_{j_k=1}^{n_k}
  e_{j_1 \dots i \dots j_k}
  =
  \sum_{j_1=1}^{n_1} \cdots \widehat{\sum_{j_\ell}} \cdots \sum_{j_k=1}^{n_k}
  e_{j_1 \dots i+1 \dots j_k}.
  \]
The \textbf{zero weight space} of degree $d$ consists of the monomials of weight zero.
\end{definition}

\begin{definition}
Let $E = ( e_{i_1 i_2 \dots i_k} )$ be an array of size $n_1 \times n_2 \times \cdots \times n_k$
with non-negative integer entries.
A \textbf{slice} of $E$ is a $(k{-}1)$-dimensional subarray obtained by fixing one subscript;
for every $\ell = 1, 2, \dots, k$ we can set $i_\ell = 1, 2, \dots, n_\ell$
and obtain $n_\ell$ slices of size $n_1 \times \cdots \widehat{n_\ell} \cdots \times n_k$.
We call $E$ an \textbf{equal parallel slice (EPS) array} if
for every $\ell = 1, 2, \dots, k$ the $n_\ell$ slices in direction $\ell$ have the same entry sum.
That is, for each $\ell$ the following sum does not depend on $j$:
  \[
  \sum_{i_1=1}^{n_1} \cdots \widehat{\sum_{i_\ell}} \cdots \sum_{i_k=1}^{n_k}
  e_{i_1 \dots j \dots i_k}.
  \]
\end{definition}

\begin{lemma}
A basis for the zero weight space in degree $d$ consists of the monomials
whose arrays of exponents are EPS arrays.
\end{lemma}

We write $W(d;a_1,\dots,a_{n-1})$ for the vector space with basis consisting of the monomials
with degree $d$ and eigenvalues $(a_1,\dots,a_{n-1})$ for $H_1, \dots, H_{n-1}$ as in Lemma \ref{weightlemma}.
In $\mathfrak{sl}_n(\mathbb{C})$ the brackets of $H_i$ and $E_j$ are given by the formulas
  \[
  [ H_i, E_j ]
  =
  \begin{cases}
  2 E_j &\text{if $i = j$} \\
  - E_j &\text{if $j = i-1$ or $j = i+1$} \\
  0 &\text{otherwise}.
  \end{cases}
  \]
It follows that the actions of $E_1, \dots, E_{n-1}$ induce the following linear maps:
  \begin{align*}
  E_1 \colon W(d;0,\dots,0) &\longrightarrow W(d;2,-1,0,\dots,0,0), 
  \\
  E_2 \colon W(d;0,\dots,0) &\longrightarrow W(d;-1,2,-1,\dots,0,0), 
  \\
  E_3 \colon W(d;0,\dots,0) &\longrightarrow W(d;0,-1,2,\dots,0,0), 
  \\  
  &\;\;\; \vdots
  \\
  E_{n-1} \colon W(d;0,\dots,0) &\longrightarrow W(d;0,0,0,\dots,-1,2).
  \end{align*}
The weights appearing on the right are the rows of the Killing-Cartan matrix,
  \[
  K^{(n-1)} = (\kappa_{ij} ),
  \qquad
  \kappa_{ij}
  =
  \begin{cases}
   2 &\text{if $i = j$} \\
  -1 &\text{if $j = i-1$ or $j = i+1$} \\
   0 &\text{otherwise}.
  \end{cases}
  \]
We write $w_1^{(n-1)}, \dots, w_{n-1}^{(n-1)}$ for the rows of $K^{(n-1)}$ and form the linear map
  \[
  E = (E_1,\dots,E_{n-1})\colon W(d;0,\dots,0) \longrightarrow \bigoplus_{i=1}^{n-1} W(d;w_i^{(n-1)}).
  \]  
We apply this to the semisimple Lie algebra \eqref{semisimple}.
We first combine the spaces $W(d;0,\dots,0)$ for each summand into the zero weight space
of Definition \eqref{zerodefinition}:
  \[
  Z
  =
  W(d;\overbrace{0,\dots,0}^{n_1-1})
  \, \cap \,
  \cdots
  \, \cap \,
  W(d;\overbrace{0,\dots,0}^{n_k-1}).  
  \]
We then combine the linear maps $E$ for each summand into the single linear map
  \begin{equation}
  \label{singlemap}
  \mathcal{E} 
  = 
  \big( E^{(n_1)}, \dots, E^{(n_k)} \big)
  \colon 
  Z 
  \longrightarrow
  \bigoplus_{\ell=1}^k
  \bigoplus_{i=1}^{n_\ell-1} 
  W\big( d; w_i^{(n_\ell-1)} \big).
  \end{equation}

\begin{theorem}
The invariant polynomials in degree $d$ for the $n_1 \times \cdots \times n_k$ array
$X = ( x_{i_1 \cdots i_k})$ 
are the (nonzero) elements of the kernel of the linear map \eqref{singlemap}.
\end{theorem}

In degree $d$ there are no monomials of weight zero unless $d$ is a multiple
of $N = \mathrm{LCM}(n_1,\dots,n_k)$; hence invariants can only exist in degrees 
$d \equiv 0$ (mod $N$).

We expect that the dimension of the zero weight space in degree $d$ 
(equivalently, the number of EPS arrays with entry sum $d$) is a polynomial in $d$.
Since we have $n_1 \cdots n_k$ exponents, 
with  one constraint on the degree and $(n_1{-}1) + \cdots + (n_k{-}1)$ constraints on the parallel slices,
we make the following conjecture.

\begin{conjecture}
Let $k$ and $n_1, n_2, \dots, n_k$ be positive integers.
The dimension of the zero weight space in degree $d$ is given by a family of polynomials of degree
  \[
  \prod_{\ell=1}^k n_\ell
  -
  \sum_{\ell=1}^k n_\ell
  +
  k
  -
  1.
  \] 
\end{conjecture}

%%%%%%%%%%%%%%%%%%%%%%%%%%%%%%%%%%%%%%%%%%%%%%%%%%%%%%%%%%%%%%%%%%%%%%%%

\section{Conclusion} \label{sectionconclusion}

Modern interest in Cayley's hyperdeterminant and its generalizations 
was revived by the famous paper of Gelfand, Kapranov and Zelevinsky \cite{GKZpaper};
see especially Proposition 1.9 on page 234.
The same authors developed this subject in great depth, using the techniques of algebraic geometry,
in their monograph \cite{GKZbook}.

A closely related topic, of great importance in applied numerical linear algebra,
is the problem of computing the rank of a $k$-dimensional array.
When $k = 2$, this problem has an efficient solution using Gaussian elimination, 
but for $k \ge 3$ it has been shown by Hastad \cite{Hastad} to be NP-complete.
A comprehensize survey on tensor rank and algorithms for tensor decomposition
has been given recently by Kolda and Bader \cite{KoldaBader}.
Cayley's hyperdeterminant was rediscovered in the 1970's by Kruskal \cite{Kruskal},
and is sometimes called Kruskal's polynomial by applied mathematicians;
see ten Berge \cite{tenBerge} and Martin \cite{Martin} for an explanation
of how it can be used to compute the rank of a $2 \times 2 \times 2$ array.
Two recent related papers are de Silva and Lim
\cite{deSilvaLim} and Stegeman and Comon \cite{StegemanComon}.

Invariant polynomials on arrays of size $2 \times 2 \times \cdots \times 2$
($k$ factors) have been studied by theoretical physicists working on quantum 
computing; see Luque and Thibon \cite{LuqueThibon1,LuqueThibon2}, Djokovic and Osterloh \cite{Djokovic}.
For the combinatorial-geometric aspects of this problem, see Huggins et al.~\cite{Huggins}.
These invariants can be regarded as noncommutative analogues of classical invariant theory
(for a survey see Dixmier \cite{Dixmier}):
the 19th century invariant theorists studied
the irreducible representations $V(d) \cong S^k V(1)$ of $\mathfrak{sl}_2(\mathbb{C})$,
and replacing the symmetric power by the full tensor power gives the vector space of arrays of size $2^k$.
It is an open problem to extend the methods of the present paper to these arrays.
It would be very useful to have a complete description of the structure of the space
of homogeneous polynomials as a sum of irreducible representations of the semisimple Lie
algebra; one possible approach to this problem has been developed by Adsul and Subrahmanyam
\cite{AdsulSubrahmanyam}.

The objects that we call equal parallel slice (EPS) arrays are examples of contingency tables, 
which are important in combinatorics and statistics.
For asymptotic formulas for the enumeration of these objects,
see Barvinok \cite{Barvinok}.
In closing, we mention the intriguing applications of Gr\"obner bases and hyperdeterminants
to mathematical genetics; see Allman and Rhodes \cite{AllmanRhodes}, especially page 146.

%%%%%%%%%%%%%%%%%%%%%%%%%%%%%%%%%%%%%%%%%%%%%%%%%%%%%%%%%%%%%%%%%%%%%%%%

\section*{Acknowledgements}

Murray Bremner was partially supported by a Discovery Grant from NSERC.
The authors thank
Tobias Pecher for reference \cite{AdsulSubrahmanyam},
Richard Brualdi for reference \cite{Barvinok}, and
Andrew Douglas for reference \cite{NSS}.

%%%%%%%%%%%%%%%%%%%%%%%%%%%%%%%%%%%%%%%%%%%%%%%%%%%%%%%%%%%%%%%%%%%%%%%%

\end{document}